\documentclass{article}
\usepackage{amsmath}
\usepackage{amsfonts}
\usepackage{amssymb}
\usepackage{amsthm}
\usepackage{cleveref}
\usepackage{fullpage}

\usepackage{enumerate}
\newtheorem{Thm}{Theorem}[section]

\newtheorem{Con}[Thm]{Conjecture}
\newtheorem{Lem}[Thm]{Lemma}
\newtheorem{Cor}[Thm]{Corollary}
\newtheorem{Pro}[Thm]{Proposition}

\makeatletter
\def\blfootnote{\xdef\@thefnmark{}\@footnotetext}
\makeatother

\theoremstyle{definition}

\theoremstyle{remark}
\newtheorem{Rem}[Thm]{\bf{Remark}}
\newtheorem{Eg}[Thm]{\bf{Example}}
\newcommand{\ConvD}{\overset{d}{\rightarrow}}
\newcommand{\ConvFDD}{\overset{f.d.d.}{\longrightarrow}}

\title{Multivariate limit theorems in the context of long-range dependence}
\author{
        Shuyang Bai \\
        Murad S. Taqqu
}

\begin{document}
\maketitle
\begin{abstract}
We study the limit law of a vector made up of normalized sums of functions of long-range dependent stationary Gaussian series. Depending on the memory parameter of the Gaussian series and on the Hermite ranks of the functions, the resulting limit law may be (a) a multivariate Gaussian process involving dependent Brownian motion marginals, or (b) a multivariate process involving dependent Hermite processes as marginals, or (c) a combination. We treat cases (a), (b) in general and case (c) when the Hermite components  involve ranks 1 and 2.
We include a conjecture about case (c) when the Hermite ranks are arbitrary.
\end{abstract}
\blfootnote{
\begin{flushleft}
\end{flushleft}
\textbf{2010 AMS Classification:} 60G18, 60F05\\
}

\section{Introduction}
A stationary time series displays long-range dependence if its auto-covariance decays slowly or if its spectral density diverges around the zero frequency. When there is long-range dependence, the asymptotic limits of various estimators are often either Brownian Motion or a Hermite process. The most common Hermite processes are fractional Brownian motion (Hermite process of order 1) and the Rosenblatt process (Hermite process of order 2), but there are Hermite processes of any order. Fractional Brownian motion is the only Gaussian Hermite process.

Most existing limit theorems involve univariate convergence, that is, convergence to a single limit process, for example, Brownian motion or a Hermite process (\cite{breuer1983central,dobrushin1979non,taqqu1979convergence}). In time series analysis, however, one often needs joint convergence, that is, convergence to a vector of processes. This is because one often needs to consider different statistics of the process jointly. See, for example, \cite{levy2011asymptotic,Rooch:2012}.
We establish a number of results involving joint convergence, and conclude with a conjecture.

Our setup is as follows.
Suppose $\{X_n\}$ is a stationary Gaussian series with mean 0, variance 1 and  regularly varying auto-covariance
\begin{equation}\label{e:LRDcov}
\gamma(n)=L(n)n^{2d-1}
\end{equation}
where $0<d<1/2$, and $L$ is a slowly varying function at infinity. This is often referred to \textit{``long-range dependence''(LRD)} or \textit{``long memory''} in the literature, and $d$ is called the \emph{memory parameter}. The higher $d$, the stronger the dependence. The slow decay (\ref{e:LRDcov}) of $\gamma(n)$ yields
\[\sum_{n=-\infty}^\infty|\gamma(n)|=\infty.\]
The case where \[\sum_{n=-\infty}^\infty|\gamma(n)|<\infty,\] is often referred to \textit{``short-range dependence'' (SRD)} or \textit{``short memory''}.  See \cite{beran1994statistics,doukhan2003theory,giraitis2009large} for more details about these notions.

We are interested  in the limit behavior of the finite-dimensional distributions (f.d.d.) of the following vector as $N\rightarrow \infty$:
\begin{equation}\label{VectorQuest}
\mathbf{V}_N(t)=\left(\frac{1}{A_j(N)}\sum_{n=1}^{[Nt]}\Big(G_j(X_n)-\mathbb{E}G_j(X_n)\Big)\right)_{j=1,\ldots,J},
\end{equation}
where $G_j$, $j=1,\ldots,J$ are nonlinear functions, $t>0$ is the time variable, and  $A_j(N)$'s are appropriate normalizations which make the variance of each component at $t=1$ tend to 1.
Observe that the same sequence $\{X_n\}$ is involved in each component of $\mathbf{V}_N$, in contrast to \cite{ho1990limiting} who consider the case $J=2$ and
$\{(X_n,Y_n)\}$ is a bivariate Gaussian vector series.

Note also that convergence in f.d.d.\ implies that our results continue to hold if one replaces the single time variable $t$ in (\ref{VectorQuest}) with a vector $(t_1,\ldots,t_J)$ which would make $\mathbf{V}_N(t_1,\ldots,t_J)$ a random field.

Depending on the memory parameter of the Gaussian series and on the Hermite ranks  of the functions (Hermite ranks are defined in Section \ref{Sec:ReviewUni}), the resulting limit law for (\ref{VectorQuest}) may be:
\begin{enumerate}[(a)]
\item a multivariate Gaussian process with dependent Brownian motion marginals,
\item or a multivariate process with dependent Hermite processes as marginals,
\item or a combination.
\end{enumerate}
We treat cases (a), (b) in general and case (c) when the Hermite components  involve ranks 1 and 2 only.
To address case (c), we apply a recent asymptotic independence theorem of Nourdin and Rosinski \cite{nourdin2012asymptotic}  of Wiener-It\^{o} integral vectors. We include a conjecture about case (c) when the Hermite ranks are arbitrary. We also prove that the Hermite processes in the limit are dependent on each other. Thus, in particular, fractional Brownian motion and the Rosenblatt process in the limit are dependent processes even though they are uncorrelated. Although our results are formulated in terms of convergence of f.d.d.\ , under some additional assumption, they extend to weak convergence in  $D[0,1]^J$(J-dimensional product space where $D[0,1]$ is the space of C\`adl\`ag functions on $[0,1]$ with the uniform metric), as noted in \cref{weakconv} at the end of Section \ref{Sec:Mult}.

The paper is structured as follows. We review the univariate results in Section \ref{Sec:ReviewUni}. In Section \ref{Sec:Mult}, we state the corresponding multivariate results. Section \ref{Sec:Proof} contains the proofs of the theorems in Section \ref{Sec:Mult}.
Appendix \ref{sec:JointInvar} shows that the different representations of the Hermite processes are also equivalent in a multivariate setting. Appendix \ref{sec:AsympIndep} refers to the results of \cite{nourdin2012asymptotic} and concerns asymptotic independence of Wiener-It\^o integral vectors.

\section{Review of the univariate results}
\label{Sec:ReviewUni}
We review first results involving (\ref{VectorQuest}) when $J=1$ in (\ref{VectorQuest}).
Assume that $G$ belongs to $L^2(\phi)$, the set of square-integrable functions with respect to the standard Gaussian measure $\phi$. This Hilbert space $L^2(\phi)$ has a complete orthogonal basis $\{H_m(x)\}_{m\ge 0}$, where $H_m$ is the  \emph{Hermite polynomial} defined as
$$
H_m(x)=(-1)^m\exp\left(\frac{x^2}{2}\right)\frac{d^m}{dx^m}\exp\left(\frac{-x^2}{2}\right),
$$
 (\cite{nourdin2012normal}, Chapter 1.4). Therefore, every function $G\in L^2(\phi)$ admits the following type of expansion:
\begin{equation}\label{HermExpan}
G=\sum_{m\ge 0}g_mH_m,
\end{equation}
where $g_m=(m!)^{-1}\int_\mathbb{R}G(x)H_m(x)d\phi(x)$.

Since $H_0(x)=1$ and since we always center the series $\{G(X_n)\}$ by subtracting its mean in (\ref{VectorQuest}),
we may always assume  $g_0=\mathbb{E}G(X_n)=0$.
The smallest index $k\ge 1$ for which $g_k\neq 0$ in the expansion (\ref{HermExpan}) is called the \emph{Hermite rank} of $G$.

Since $\{X_n\}$ is a stationary Gaussian series, it has the following spectral representation
\begin{equation}\label{SpecRep}
X_n=\int_{\mathbb{R}}e^{inx}dW(x),
\end{equation}
where $W$ is the complex Hermitian ($W(A)=\overline{W(-A)}$) Gaussian random measure specified by $\mathbb{E}W(A)\overline{W(B)}=F(A\cap B)$. The measure $F$ is called the \emph{spectral distribution} of $\{X_n\}$, is also called the \emph{control measure} of $W$, and is defined by $\gamma(n)=\mathbb{E}X_nX_0=\int_{\mathbb{R}}e^{inx}dF(x)$ (See \cite{lifshits2012lectures}, Chapter 3.2).

Multiple Wiener-It\^{o} integrals (\cite{major1981multiple})
\begin{equation}\label{WienerItoInt}
I_m(K)=\int''_{\mathbb{R}^m}K(x_1,\ldots,x_m)dW(x_1)\ldots dW(x_m)
\end{equation}
where $\int_{\mathbb{R}^m} |K(x_1,\ldots,x_m)|^2dF(x_1)\ldots dF(x_m)<\infty $,
 play an important role because of the following connection between Hermite polynomials and multiple Wiener-It\^{o} integrals (\cite{nourdin2012normal} Theorem 2.7.7):
\begin{equation}\label{Herm<->Int}
H_m(X_n)=\int''_{\mathbb{R}^m}e^{in(x_1+\ldots+x_m)}dW(x_1)\ldots dW(x_m),
\end{equation}
where the double prime $''$ indicates that one doesn't integrate on the hyper-diagonals $x_j=\pm x_k$, $j\neq k$.
Throughout this paper, $I_m(.)$ denotes a $m$-tuple Wiener-It\^{o} integral of the type in (\ref{WienerItoInt}).

We  now recall some well-known univariate results:

\begin{Thm}\label{CLT}\textbf{(SRD Case.)}
Suppose the memory parameter $d$ and the Hermite rank $k\ge 1$ of $G$ satisfy
\[
0<d<\frac{1}{2}(1-\frac{1}{k}).
\]
Then
\[
\frac{1}{A(N)}\sum_{n=1}^{[Nt]}G(X_n)\overset{f.d.d.}{\longrightarrow}B(t),
\]
where $B(t)$ is a standard Brownian Motion, `` $\overset{f.d.d.}{\longrightarrow}$ '' denotes convergence in finite-dimensional distributions along the time variable $t>0$,
$A(N)\propto N^{1/2}$ is a normalization factor such that
\[
\lim_{N\rightarrow \infty}\mathrm{Var} \left(\frac{1}{A(N)}\sum_{n=1}^{N}G(X_n)\right) = 1.
\]
\end{Thm}

\begin{Rem}
It can indeed be shown that in the setting of \cref{CLT},
\begin{equation}\label{VarGrowthSRD}
\mathrm{Var}\left(\sum_{n=1}^{N}G(X_n)\right)\sim \sigma^2N,
\end{equation}
where
\begin{equation}\label{sigma}
\sigma^2=\sum_{m=k}^\infty g_m^2m!\sum_{n=-\infty}^\infty\gamma(n)^m.
\end{equation}
Recall that the $g_m$'s are the coefficients of the Hermite expansion of $G$, and $\gamma$ is the auto-covariance function of $\{X_n\}$.
\end{Rem}

\begin{Rem}\label{summability of autocov}
The condition $0<d<\frac{1}{2}(1-\frac{1}{k})$ can be replaced with a weaker condition $\sum_{n=-\infty}^\infty |\gamma(n)|^k<\infty$ or equivalently $\sum_{n=-\infty}^\infty |\gamma_G(n)|<\infty$, where $\gamma_G(n)$ is the auto-covariance function of $\{G(X_n)\}$. See Theorem 4.6.1 in \cite{giraitis2009large}. If $d=\frac{1}{2}(1-\frac{1}{k})$ but as $N\rightarrow\infty$, $\sum_{n=-N}^N|\gamma(n)|^k=\sum_{-N}^Nn^{-1}|L(n)|^k=:L^*(N)\rightarrow\infty$ is slowly varying, then one still gets convergence to Brownian motion (Theorem 1' of \cite{breuer1983central}), but with the normalization $A(N)\propto \left(NL^*(N)\right)^{1/2}$. For example, if the slowly varying function in (\ref{e:LRDcov}) is $L(n)\sim c>0$, then $A(N)\propto(N\ln N)^{1/2}$.
\end{Rem}

The original proof of \cref{CLT} (\cite{breuer1983central}) was done by a method of moments  using the so-called diagram formulas (\cite{peccati2011wiener}), which provide explicit ways to compute the cumulants of Hermite polynomials of Gaussian random variable. Recently, a remarkable technique for establishing central limit theorems of multiple Wiener-It\^{o} integral was found by \cite{nualart2005central,peccati2005gaussian}, whereby in the multiple Wiener-It\^{o} integral setting, convergence of the fourth moment, or some equivalent easier-to-check  condition,  implies directly the Gaussian limit. See Theorem 7.2.4 in \cite{nourdin2012normal} for a proof in the case $t=1$.

\begin{Thm}\label{NCLT}\textbf{(LRD Case.)}
Suppose that the memory parameter $d$ and the Hermite rank $k\ge 1$ of $G$ satisfy
\[
\frac{1}{2}(1-\frac{1}{k})<d<\frac{1}{2}.
\]
Then
\[
\frac{1}{A(N)}\sum_{n=1}^{[Nt]}G(X_n)\overset{f.d.d.}{\longrightarrow}Z_{d}^{(k)}(t):=I_{k}(f_{k,d}^{(t)}),
\]
where
the control measure of $I_{k}(.)$ is Lebesgue,
 $A(N)\propto N^{1+(d-1/2)k}L(N)^{k/2}$ is a normalization such that
\[
\lim_{N\rightarrow\infty} \mathrm{Var}\left(\frac{1}{A(N)}\sum_{n=1}^{N}G(X_n)\right)=1,
\]
and
\[
f_{k,d}^{(t)}(x_1,\ldots,x_k)=b_{k,d}\frac{e^{it(x_1+\ldots+x_k)}-1}{i(x_1+\ldots+x_k)}|x_1|^{-d}\ldots |x_k|^{-d},
\]
where
\[b_{k,d}=\left(\frac{\left(k(d-1/2)+1\right)\left(2k(d-1/2)+1\right)}{k!\left(2\Gamma(1-2d)\sin(d\pi)\right)^k}\right)^{1/2}\]
is the normalization constant to guarantee unit variance for $Z^{(k)}(1)$.
\end{Thm}
For a proof, see \cite{dobrushin1979non} and \cite{pipiras2010regularization}.
The process $Z_{d}^{(k)}(t)$ appearing in the limit is called a \emph{Hermite process}.

\begin{Rem}
It can indeed be shown that in the setting of \cref{NCLT},
\begin{equation}\label{VarGrowthLRD}
\mathrm{Var}\left(\sum_{n=1}^{N}G(X_n)\right)= L_{G}(N)N^{2d_G+1}
\end{equation}
for some slowly varying function $L_{G}(N)\propto L(N)^k$ and $d_G=(d-1/2)k+1/2$
 (see e.g.\ (3.3.8) in \cite{giraitis2009large}). Since $d<1/2$, increasing the Hermite rank $k$ decreases the memory parameter $d_G$, hence decreases the dependence. Note that if $k\ge 2$, then the variance growth of $\{G(X_n)\}$ in (\ref{VarGrowthLRD}) is slower than the variance growth of $\{X_n\}$,
$\mathrm{Var}(\sum_{n=1}^{N}X_n)= L_0(N)N^{2d+1}$ for some slowly varying function $L_0$, but is always faster than the variance growth $\sigma^2N$ in the SRD case in (\ref{VarGrowthSRD}).
\end{Rem}

The process $Z_{d}^{(1)}(t)$, $t\ge 0$ is a Gaussian process called \emph{fractional Brownian motion}, and
$Z_{d}^{(2)}(t)$, $t\ge 0$ is a non-Gaussian process called \emph{Rosenblatt process}.
The Hermite processes $Z_{d}^{(k)}(t)$ are all so-called \emph{self-similar processes} (\cite{embrechts2002selfsimilar}).


\section{Multivariate convergence results}
\label{Sec:Mult}
Our  aim is to  study the limit of (\ref{VectorQuest}), and in particular, to extend
\cref{CLT} (SRD) and \cref{NCLT} (LRD) to a multivariate setting.

Suppose that for each $j=1,\ldots,J$, the function the functions $G_j$ in (\ref{VectorQuest}) belongs to $L^2(\phi)$, has Hermite rank $k_j$ and admits Hermite expansion $\sum_{m=k_j}^\infty g_{m,j}H_m$ (see (\ref{HermExpan})).

We start with the pure SRD case where every component $\{G_j(X_n)\}$ of $\mathbf{V}_N(t)$ in (\ref{VectorQuest}) is SRD.

\begin{Thm}\label{Pure SRD}\textbf{(Pure SRD Case.)}
If the memory parameter $d$ is small enough so that all $\{G_j(X_n)\}, j=1,\ldots,J$ are SRD, that is,
\[
d<\frac{1}{2}(1-\frac{1}{k_j}),~ j=1,\ldots,J,
\]
then in (\ref{VectorQuest})
\[
\mathbf{V}_N(t)\overset{f.d.d.}{\longrightarrow}\mathbf{B}(t)
\]
as $N\rightarrow \infty$, where the normalization $A_j(N)\propto N^{1/2}$ is such that for $j=1,\ldots,J$,
\begin{equation}\label{e:Norm-SRD}
\lim_{N\rightarrow\infty}\mathrm{Var}\left(\frac{1}{A_j(N)}\sum_{n=1}^{N}G_j(X_n)\right)=1.
\end{equation}
Here $\mathbf{B}(t)=(B_1(t),\ldots,B_J(t))$ is a multivariate Gaussian process with standard Brownian motions as marginals, and where the cross-covariance between two components is
\begin{align}\label{CovStruct BM}
 \mathrm{Cov}\left(B_{j_1}(t_1),B_{j_2}(t_2)\right)
 &=\lim_{N\rightarrow\infty}\mathrm{Cov}(V_{N,j_1}(t_1),V_{N,j_2}(t_2))\notag\\
 &=(t_1\wedge t_2) \left[\frac{1}{\sigma_{j_1}\sigma_{j_2}}\sum_{m=k_{j_1}\vee k_{j_2}}^\infty g_{m,j_1}g_{m,j_2}m!\sum_{n=-\infty}^{\infty}\gamma(n)^m\right]
\end{align}
where
\begin{equation}\label{sigma_j}
\sigma_j^2=\sum_{m=k_j}^\infty g_{m,j}^2 m! \sum_{n=-\infty}^\infty \gamma(n)^m.
\end{equation}
\end{Thm}
This theorem is proved in Section \ref{Sec:Pure SRD}.

\begin{Eg}
Assume that the auto-covariance function $\gamma(n)\sim n^{2d-1}$ as $n\rightarrow \infty$.
Let $J=2$, $G_1(x)=aH_2(x)+bH_3(x)=bx^3+ax^2-3bx-a$, $G_2(x)=cH_3(x)=cx^3-3cx$ and $0<d<1/4$.
Then in (\ref{sigma_j}),
$\sigma_1^2=2a^2\sum_{n=-\infty}^\infty\gamma(n)^2+6b^2\sum_{n=-\infty}^\infty \gamma(n)^3$, $\sigma_2^2=6c^2\sum_{n=-\infty}^\infty\gamma(n)^3$, and
\[
\left(\frac{1}{N^{1/2}}\sum_{n=1}^{[Nt]}(X_n^2-1) , \frac{1}{N^{1/2}}\sum_{n=1}^{[Nt]}(X_n^3-3X_n)  \right) \ConvFDD
\left(\sigma_1B_1(t),\sigma_2B_2(t)\right),
\]
where the Brownian motions $B_1$ and $B_2$ have the covariance structure:
\[
\mathrm{Cov}\left(B_{1}(t_1),B_{2}(t_2)\right)=6b
\frac{t_1\wedge t_2}{\sigma_{1}\sigma_{2}}\sum_{n=-\infty}^\infty\gamma(n)^3.
\]
$B_1$ and $B_2$ are independent when $b=0$.
\end{Eg}

Next we consider the case where every component $\{G_j(X_n)\}$ of $\mathbf{V}_N(t)$ in (\ref{VectorQuest}) is LRD.

\begin{Thm}\label{Pure LRD}\textbf{(Pure LRD  Case.)}
If the memory parameter $d$ is large enough so that all $G_j(X_n), j=1,\ldots,J$ are LRD, that is,
\[
d>\frac{1}{2}(1-\frac{1}{k_j}), ~j=1,\ldots,J,
\]
then in (\ref{VectorQuest}),
\begin{equation}\label{Pure LRD Limit}
\mathbf{V}_N(t)\overset{f.d.d.}{\longrightarrow}\mathbf{Z}_d^{\mathbf{k}}(t):=\Big(I_{k_1}(f_{k_1,d}^{(t)}),\ldots,I_{k_J}(f_{k_J,d}^{(t)})\Big),
\end{equation}
where the normalization $A_j(N)\propto N^{1+(d-1/2)k_j}L(N)^{k_j/2} $ is such that for $j=1,\ldots,J$,
\begin{equation}\label{e:Norm-LRD}
\lim_{N\rightarrow\infty}\mathrm{Var}\left(\frac{1}{A_j(N)}\sum_{n=1}^{N}G_j(X_n)\right)=1.
\end{equation}
Each component of $\mathbf{Z}_d^\mathbf{k}(t):=\left(Z_d^{(k_1)}(t),\ldots,Z_d^{(k_J)}(t)\right)$ is a standard Hermite process, and $I_{k}(.)$ denotes $k$-tuple Wiener-It\^{o} integral with respect to a common complex Hermitian Gaussian random measure $W$ with Lebesgue control measure, and
\begin{equation}\label{Hermite Kernel}
f_{k,d}^{(t)}(x_1,\ldots,x_k)=b_{k,d}\frac{e^{it(x_1+\ldots+x_k)}-1}{i(x_1+\ldots+x_k)}|x_1|^{-d}\ldots |x_k|^{-d},
\end{equation}
where $b_{k,d}$'s are the same  normalization constants as in \cref{NCLT}.
\end{Thm}
This theorem is proved in Section \ref{Sec:Pure LRD}.

\begin{Eg}
Assume that auto-covariance function $\gamma(n)\sim n^{2d-1}$ as $n\rightarrow \infty$.
Let $J=2$, $G_1(x)=H_1(x)=x$,
 $G_2(x)=H_2(x)=x^2-1$, $1/4<d<1/2$, then
\[
\left(\frac{1}{N^{1/2+d}}\sum_{n=1}^{[Nt]}X_n , \frac{1}{N^{2d}}\sum_{n=1}^{[Nt]}(X_n^2-1)  \right) \ConvFDD
\left(\frac{1}{d(2d+1)}Z_d^{(1)}(t),\frac{1}{d(4d-1)}Z_d^{(2)}(t)\right),
\]
where the standard fractional Brownian motion  $Z_d^{(1)}(t)$ and standard Rosenblatt process $Z_d^{(2)}(t)$ share the same random measure in the Wiener-It\^{o} integral representation. The components $Z_d^{(1)}$ and $Z_d^{(2)}$ are uncorrelated but dependent as stated below.
\end{Eg}

In \cref{Pure LRD}, the marginal Hermite processes $Z_d^{(k_1)}(t)=I_{k_1}(f_{k_1,d}^{(t)}),\ldots,Z_d^{(k_J)}(t)=I_{k_J}(f_{k_J,d}^{(t)})$ are dependent on each other.
To prove this, we use a different representation of the Hermite process, namely, the positive half-axis representation given in (\ref{eq:PosAxiRep}).

\begin{Pro}\label{Pro:HermDep}
The marginal Hermite processes $Z_d^{(k_1)},\ldots,Z_d^{(k_J)}$ involved in \cref{Pure LRD} are dependent.
\end{Pro}
\begin{proof}
From  \cite{ustunel1989independence}, we have the following criterion for  the independence of multiple Wiener-It\^o integrals:
suppose that symmetric $g_1\in L^2(\mathbb{R}^p_{+})$ and $g_2\in L^2(\mathbb{R}^q_{+})$.
 Then $I_p(g_1)$ and $I_q(g_2)$ ($p,q\ge 1$) are independent if and only if
\[
g_1\otimes_1 g_2:=\int_{\mathbb{R}_+}g_1(x_1,\ldots,x_{p-1},u)g_2(x_p,\ldots,x_{p+q-2},u)du =0 \text{ in } L^2(\mathbb{R}_+^{p+q-2}).
\]
We shall apply this criterion to the positive half-axis integral representation (\ref{eq:PosAxiRep}) of Hermite processes (see also \cite{pipiras2010regularization}):
\[
Z_d^{(k)}(t)=c_{k,d}I_k\left(g_{k,d}^{(t)}(x_1,\ldots,x_k)\right):=c_{k,d}\int'_{\mathbb{R}_+^k}\left[\int_0^t\prod_{j=1}^kx_j^{-d}(1-sx_j)_+^{d-1}ds\right]dB(x_1)\ldots dB(x_k),
\]
where $B$ is Brownian motion, the prime $'$ indicates the exclusion of diagonal $x_j=x_k, j\neq k$ and $c_{k,d}$ is some normalization constant. In fact, for a vector made up of Hermite processes sharing the same random measure in their Wiener-It\^o integral representation, the joint  distribution does not change when switching from one representation of Hermite process to another. See Appendix \ref{sec:JointInvar}.

One can then see (let $t=1$ and thus $g_{k,d}:=g_{k,d}^{(1)}$ for simplicity) that for all $(x_1,\ldots,x_{p+q-2})\in \mathbb{R}_+^{p+q-2}$:
\begin{align*}
 &(g_{p,d}\otimes_1g_{q,d})(x_1,\ldots,x_{p+q-2})\\
=&\int_{\mathbb{R}_+}\left(\int_0^1 \prod_{j=1}^{p-1}x_j^{-d}(1-sx_j)_+^{d-1} u^{-d}(1-su)_+^{d-1}ds\int_0^1 \prod_{j=p}^{p+q-2}x_j^{-d}(1-sx_j)_+^{d-1} u^{-d}(1-su)_+^{d-1}ds \right) du>0
\end{align*}
because every term involved in the integrand is positive.
\end{proof}

\cref{Pure SRD} and \cref{Pure LRD} describe the convergence of $\mathbf{V}_N(t)$ in (\ref{VectorQuest}) when the $\{G_j(X_n)\}$, $j=1,\ldots,J$ are all purely SRD or purely LRD. However, when the components in $\mathbf{V}_N(t)$ are mixed, that is, some of them are SRD and some of them are LRD, it is not immediately clear what the limit behavior is and also what the inter-dependence structure between the SRD and LRD limit components is.
We show that the SRD part and LRD part are asymptotically independent so that one could join the limits of \cref{Pure SRD} and \cref{Pure LRD} together, in the case when the $G_j$'s in the LRD part only involve the 2 lowest Hermite ranks, namely, $k=1$ or $k=2$. This is stated in the next theorem where the letter ``S'' refers to the SRD part and ``L'' to the LRD part.

\begin{Thm}\label{SRD&LRD}\textbf{(Special SRD and LRD Mixed Case.)}
Separate the SRD and LRD parts of $\mathbf{V}_N(t)$ in (\ref{VectorQuest}), that is, let
$\mathbf{V}_N(t)=\left(\mathbf{S}_N(t),\mathbf{L}_N(t)\right)$, where
\begin{align}
\mathbf{S}_N(t)&=\left(\frac{1}{A_{1,S}(N)}\sum_{n=1}^{[Nt]}G_{1,S}(X_n),\ldots,\frac{1}{A_{J_S,S}(N)}\sum_{n=1}^{[Nt]}G_{J_S,S}(X_n)\right),\label{SRD Part}\\
\mathbf{L}_N(t)&=\left(\frac{1}{A_{1,L}(N)}\sum_{n=1}^{[Nt]}G_{1,L}(X_n),\ldots,\frac{1}{A_{J_L,L}(N)}\sum_{n=1}^{[Nt]}G_{J_L,L}(X_n)\right),\label{LRD Part}
\end{align}
where $G_{j,S}$ has Hermite rank $k_{j,S}$, and  $G_{j,L}$ has  Hermite rank $k_{j,L}$, $A_{j,S}\propto N^{1/2}$ and\\ $A_{j,L}\propto N^{1+(d-1/2)k_{j,L}}L(N)^{k_{j,L}/2}$ are the correct normalization factors such that for $j=1,\ldots,J_S$ and $j=1,\ldots,J_L$ respectively,
\begin{equation}\label{e:Norm-SRD&LRD}
\lim_{N\rightarrow\infty}\mathrm{Var}\left(\frac{1}{A_{j,S}(N)}\sum_{n=1}^{N}G_{j,S}(X_n)\right)=1, \quad
\lim_{N\rightarrow\infty}\mathrm{Var}\left(\frac{1}{A_{j,L}(N)}\sum_{n=1}^{N}G_{j,L}(X_n)\right)=1.
\end{equation}
In addition,
\begin{equation}\label{Mixed Condition}
\frac{1}{2}(1-\frac{1}{k_{j_L,L}})<d<\frac{1}{2}(1-\frac{1}{k_{j_S,S}})\quad\text{for all  } j_S=1,\ldots,J_S,~j_L=1,\ldots,J_L,
\end{equation}
where we allow arbitrary values for $k_{j,S}$ but only $k_{j,L}=1$ or $2$.
(Condition (\ref{Mixed Condition}) makes all $\{G_{j,S}(X_n)\}$ SRD and all $\{G_{j,L}(X_n)\}$ LRD.)

Then we have
\begin{equation}\label{BM FBM Limit}
(\mathbf{S}_N(t), \mathbf{L}_N(t))\ConvFDD (\mathbf{B}(t),\mathbf{Z}_d^{(\mathbf{k}_L)}(t)),
\end{equation}
where the multivariate Gaussian process $\mathbf{B}(t)$ is given in (\ref{Pure SRD}) and the multivariate standard Hermite process $\mathbf{Z}_d^{(\mathbf{k}_L)}(t)$ is given in (\ref{Pure LRD}). Moreover, the vectors $\mathbf{B}(t)$ and $\mathbf{Z}_d^{(\mathbf{k}_L)}(t)$ are independent.
\end{Thm}
This theorem is proved in Section \ref{Sec:SRD&LRD}. Observe that while $\mathbf{B}(t)$ is made up of correlated Brownian motions, it follows from \cref{SRD&LRD} that if $\mathbf{Z}_d^{(k)}(t)$ contains fractional Brownian motion as a component, then the fractional Brownian motion will be independent of any Brownian motion component of $\mathbf{B}(t)$.

We conjecture the following:
\begin{Con}
\cref{SRD&LRD} holds without the restriction that $k_{j,L}$ be $1$ or $2$.
\end{Con}

\begin{Eg}
Assume that the auto-covariance function $\gamma(n)\sim n^{2d-1}$ as $n\rightarrow\infty$.
Let $J=2$, $G_1(x)=H_2(x)=x^2-1$,
 $G_2(x)=H_3(x)=x^3-3x$, $1/4<d<1/3$,
 then $\sigma^2=6\sum_{n=-\infty}^\infty\gamma(n)^3$ and
\[
\left(\frac{1}{N^{2d}}\sum_{n=1}^{[Nt]}(X_n^2-1) , \frac{1}{N^{1/2}}\sum_{n=1}^{[Nt]}(X_n^3-3X_n)  \right) \ConvFDD
\left(\frac{1}{d(4d-1)}Z_d^{(2)}(t),
\sigma B(t)\right).
\]
where the standard Rosenblatt process  $Z_d^{(2)}(t)$ and the standard Brownian motion $B(t)$ are independent.
\end{Eg}

The proof of \cref{SRD&LRD} is based a recent result in \cite{nourdin2012asymptotic} which characterizes the asymptotic moment-independence of series of multiple Wiener-It\^{o} integral vectors. We also note that in Proposition 5.3 (2) of \cite{nourdin2012asymptotic} , a special case of \cref{SRD&LRD} with $J_S=J_L=1$ and LRD part involving Hermite rank $k_{1,L}=2$ is treated.
To go from moment-independence to independence, however, requires moment-determinancy of the limit, which we know holds when the Hermite rank $k=1,2$, that is, in the Gaussian and Rosenblatt  cases. If some other \emph{Hermite distribution} (marginal distribution of Hermite process) $Z_d^{(k)}$ ($k\ge 3$) is moment-determinate, then we will allow $k_{j,L}=k$ in \cref{SRD&LRD}. So to this end, the moment-problem of general Hermite distributions is of great interest.

\begin{Rem}
As mentioned in \cref{summability of autocov}, the border case $d_j=\frac{1}{2}(1-\frac{1}{k_j})$ often leads to convergence to Brownian motion as well. In fact, \cref{Pure SRD} and \cref{SRD&LRD} continue to hold if we extend the definition of SRD to the case whenever the limit is Brownian motion regardless of the normalization.
\end{Rem}

In \cref{Pure SRD}, \cref{Pure LRD} and \cref{SRD&LRD} we stated the results only in terms of convergence in finite-dimensional distributions, but in fact they hold under weak convergence in $D[0,1]^J$ (J-dimensional product space where $D[0,1]$ is the space of C\`adl\`ag functions on $[0,1]$ with the uniform metric). If one can check that every component of $\mathbf{V}_N(t)$ is tight, then the vector $\mathbf{V}_N(t)$ is tight:
\begin{Lem}
Univariate tightness in $D[0,1]$ implies multivariate tightness in $D[0,1]^J$.
\end{Lem}
\begin{proof}
Suppose every component $X_{j,N}$ (a random element in $S=D[0,1]$ with uniform metric $d$) of the J-dimensional random element $\mathbf{X}_N$ is tight, that is, given any $\epsilon>0$, there exists a compact set $K_j$ in $D[0,1]$, so that for all $N$ large enough:
\[
P\left(X_{j,N}\in K^c\right)<\epsilon
\]
where $K_j^c$ denotes the complement of $K_j$. If $K=K_1\times\ldots\times K_J$, then $K$ is compact in the product space $S^J$. We can associate $S^J$ with any compatible metric, e.g., for $\mathbf{X},\mathbf{Y}\in S^J$,
\[d_m(\mathbf{X},\mathbf{Y}):=\max_{1\le j\le J}(d(X_1,Y_1),\ldots, d(X_J,Y_J)).\]
The sequence $\mathbf{X}_N$ is tight on $D[0,1]^J$ since
\begin{align*}
&P\left(\mathbf{X}_{N}\in K^c\right)
=P(\cup_{j=1}^J \{X_{j,N}\in K_j^c\})\le \sum_{j=1}^J P(X_{j,N}\in K_j^c)< J\epsilon .
\end{align*}
\end{proof}
The univariate tightness is shown in \cite{taqqu1979convergence} for the LRD case. The tightness for the SRD case was considered in \cite{chambers1989central} p.~328 and holds under the following additional assumption, that $\{G(X_n)\}$ is SRD, with
\begin{equation}\label{e:tightSRD}
\sum_{k=1}^\infty 3^{k/2}(k!)^{1/2}|g_k|<\infty,
\end{equation}
where $g_k$ is the $k$-th coefficient of Hermite expansion (\ref{HermExpan}) of $G$. Observe that (\ref{e:tightSRD})  is a strengthening of the basic condition:
$\mathbb{E}[G(X_0)^2]=\sum_{k=1}k!g_k^2<\infty$. Hence we have:
\begin{Thm}\label{weakconv}
Suppose that condition (\ref{e:tightSRD}) holds for the short-range dependent components. Then the convergence in  \cref{Pure SRD}, \cref{Pure LRD} and \cref{SRD&LRD} holds as weak convergence in $D[0,1]^J$.
\end{Thm}

Condition (\ref{e:tightSRD}) is satisfied in the important special case where $G$ is a polynomial of finite order.

\section{Proofs of the multivariate convergence results}
\label{Sec:Proof}
\subsection{Proof of \cref{Pure SRD}}
\label{Sec:Pure SRD}
We start with a number of lemmas. The first yields the limit covariance structure in (\ref{CovStruct BM}).
\begin{Lem}\label{ComputeCov}
Assume that $\sum_{n}|\gamma(n)|^m<\infty$, then as $N\rightarrow \infty$:
\begin{equation} \label{e:gam1}
\frac{1}{N}\sum_{n_1=1}^{[Nt_1]}\sum_{n_2=1}^{[Nt_2]}\gamma(n_1-n_2)^m\rightarrow
(t_1\wedge t_2)\sum_{n=-\infty}^{\infty}\gamma(n)^m.
\end{equation}
\end{Lem}

\begin{proof}
Denote the left-hand side of (\ref{e:gam1}) by $S_N$. Let $a=t_1\wedge t_2$, and $b=t_1\vee t_2$, and
\[
S_{N,1}=\frac{1}{N}\sum_{n_1=1}^{[Na]}\sum_{n_2=1}^{[Na]}\gamma(n_1-n_2)^m,\qquad
S_{N,2}=\frac{1}{N}\sum_{n_1=1}^{[Na]}\sum_{n_2=[Na]+1}^{[Nb]}\gamma(n_1-n_2)^m,
\]
so $S_N=S_{N,1}+S_{N,2}$. We have as $N\rightarrow\infty$,
\[
S_{N,1}=a\sum_{n_1=-[Na]+1}^{[Na]-1}\frac{[Na]-|n|}{Na}\gamma(n)^m\rightarrow a \sum_{n=-\infty}^{\infty}\gamma(n)^m.
\]
We hence need to show that $S_{N,2}\rightarrow 0$. Let $c(n)=\gamma(n)^m$, then

\[
S_{N,2}\le\frac{1}{N}\sum_{n_1=1}^{[Na]}\sum_{n_2=[Na]+1}^{[Nb]}|c(n_2-n_1)|=\frac{1}{N}\sum_{n_1=1}^{[Na]}c_{N,n_1}=\int_0^a f_N(u)du,
\]
where
\[
c_{N,n_1}:=
\sum_{n_2=[Na]+1}^{[Nb]}|c(n_2-n_1)|
=\sum_{n_2=1}^{[Nb]-[Na]}|c([Na]+n_2-n_1)|,
\]
and for $u\in(0,a)$,
\begin{eqnarray*}
f_N(u):&=&
\sum_{n_1=1}^{[Na]}c_{N,n_1}\mathbf{1}_{[\frac{n_1-1}{N},\frac{n_1}{N})}(u) \\
  &=&\sum_{n_2=1}^{[Nb]-[Na]}\sum_{n_1=1}^{[Na]}
    |c([Na]+n_2-n_1)|\mathbf{1}_{[\frac{n_1-1}{N},\frac{n_1}{N})}(u) \\
&=&\sum_{n_2=1}^{[Nb-Na]}|c([Na]-[Nu]-1+n_2)|.
\end{eqnarray*}
Now observe that $f_N(u)\le \sum_{n=-\infty}^\infty |c(n)|
= \sum_{n=-\infty}^\infty |\gamma(n)|^m
<\infty$
and that
 $[Na]-[Nu]\rightarrow \infty$ as $N\rightarrow \infty$ . Applying the Dominated Convergence Theorem, we deduce $f_N(u)\rightarrow 0$ on $(0,a)$. Applying the Dominated Convergence Theorem again, we
 conclude that
$S_{N,2}\rightarrow 0$.
\end{proof}

Now we introduce some notations, setting for $G\in L^2(\phi)$,
\begin{equation}\label{e:sumG}
S_{N,t}(G):=\frac{1}{\sqrt{N}}\sum_{n=1}^{[Nt]}G(X_n).
\end{equation}
The Hermite expansion of each $G_j$ is
\begin{equation}\label{e:Hexp}
G_j=\sum_{m=k_j}^\infty g_{m,j}H_m
\end{equation}
if $G_j$ has Hermite rank $k_j$.
Since we are in the pure SRD case, we have as in \cref{summability of autocov}, that the auto-covariance function $\gamma(n)$ of $\{X_n\}$
\[
\sum_{n=-\infty}^\infty |\gamma(n)|^{k_j}<\infty,\quad \text{for } j=1,\ldots,J.
\]

The following lemma states that it suffices to replace a general $G_j$ with a finite linear combination of Hermite polynomials:
\begin{Lem}\label{ReductCLT}
If \cref{Pure SRD} holds with a finite linear combination of Hermite polynomials $G_j=\sum_{m=k_j}^{M} a_{m,j}H_m$ for any $M\ge \max_j(k_j)$ and any $a_{m,j}$, then it also holds for any $G_j\in L^2(\phi)$.
\end{Lem}

\begin{proof}
First we obtain an $L^2$ bound for $S_{N,t}(H_m)$.
By $\mathbb{E}H_m(X)H_m(Y)=m!\mathbb{E}(XY)^m$ (Proposition 2.2.1 in \cite{nourdin2012normal}), for $m\ge 1$,

\begin{align}
\mathbb{E}(S_{N,t}(H_m))^2&=\frac{1}{N}\sum_{n_1,n_2=1}^{[Nt]}\mathbb{E}H_m(X_{n_1})H_m(X_{n_2})
=\frac{m!}{N}\sum_{n_1,n_2=1}^{[Nt]}\gamma(n_1-n_2)^m\notag\\
&=tm!\sum_{n=1-[Nt]}^{[Nt]-1}\frac{[Nt]-|n|}{Nt}\gamma(n)^m
\le tm!\sum_{n=-\infty}^\infty|\gamma(n)|^m\label{BoundSingleHerm}.
\end{align}

Next, fix any $\epsilon>0$. By (\ref{BoundSingleHerm}) and
$\|G\|^2_{L^2(\phi)}=\sum_{m=0}^\infty g_m^2m!$, for $M$ large enough, one has
\begin{align*}
&\mathbb{E}|S_{N,t}(G_j)-S_{N,t}(\sum_{m={k_j}}^{M}g_{m,j}H_m)|^2
=\mathbb{E}| S_{N,t}(\sum_{m=M+1}^\infty g_{m,j}H_m)|^2\\
=&\sum_{m=M+1}^\infty g_{m,j}^2 \mathbb{E}(S_{N,t}(H_m))^2
\le t\sum_{n=-\infty}^\infty |\gamma(n)|^{k_j} \sum_{m=M+1}^\infty g_{m,j}^2m!
\le \epsilon t.
\end{align*}
Therefore, the $J$-vector
\[
\mathbf{V}_{N,M}(t)=\left(S_{N,t}(\sum_{m=k_1}^{M}g_{m,1}H_m),\ldots,S_{N,t}(\sum_{m=k_J}^{M}g_{m,j}H_m)\right)
\]
satisfies $\limsup_N\mathbb{E}||\mathbf{V}_{N,M}(t)-\mathbf{V}_N(t)|^2\le J\epsilon t$, and thus
\[
\lim_M \limsup_N \mathbb{E}|\mathbf{V}_{N,M}(t)-\mathbf{V}_N(t)|^2=0.
\]
By assumption, we have as $N\rightarrow \infty$ $\mathbf{V}_{N,M}(t)\ConvFDD \mathbf{B}_M(t)=\left(B_{M,1},\ldots,B_{M,J}\right)$, where the multivariate Gaussian $\mathbf{B}_M(t)$ has (scaled) Brownian motions as marginals with a covariance structure computed  using \cref{ComputeCov} as follows:
\begin{align*}
\mathbb{E}(B_{M,j_1}(t_1)B_{M,j_2}(t_2))
&=\lim_{N\rightarrow\infty}\mathbb{E}\left( S_{N,t_1}(\sum_{m=k_{j_1}}^{M}g_{m,j_1}H_m) S_{N,t_2}(\sum_{m=k_{j_2}}^{M}g_{m,j_2}H_m) \right)\\
&=\lim_{N\rightarrow\infty}\sum_{m=k_{j_1}\vee k_{j_2}}^Mg_{m,j_1}g_{m,j_2}m!\sum_{n_1=1}^{[Nt_1]}\sum_{n_2=1}^{[Nt_2]}\gamma(n_1-n_2)^m\\
&=(t_1\wedge t_2)\sum_{m=k_{j_1}\vee k_{j_2}}^Mg_{m,j_1}g_{m,j_2}m!\sum_{n=-\infty}^{\infty}\gamma(n)^m.
\end{align*}

Furthermore,  as $M\rightarrow\infty$, $\mathbf{B}_M(t)$ tends in f.d.d.\ to $\mathbf{B}(t)$, which is a multivariate  Gaussian process with the following covariance structure:
\[
\mathbb{E}(B_{j_1}(t_1)B_{j_2}(t_2))=(t_1\wedge t_2)\sum_{m=k_{j_1}\vee k_{j_2}}^\infty g_{m,j_1}g_{m,j_2}m!\sum_{n=-\infty}^{\infty}\gamma(n)^m.
\]

Therefore, applying the triangular argument in \cite{patrick1999convergence} Theorem 3.2, we have
\[
\mathbf{V}_N(t)\ConvFDD \mathbf{B}(t).
\]
\end{proof}

The proof of \cref{Pure SRD} about the pure SRD case relies on \cite{nourdin2012normal} Theorem 6.2.3, which says that for multiple Wiener-It\^{o} integrals, univariate convergence to normal random variables implies joint convergence to a multivariate normal. We state it as follows:

\begin{Lem}\label{MultiCLT}
Let $J\ge 2$ and $k_1,\ldots,k_j$ be some fixed positive integers. Consider vectors
\[
\mathbf{V}_N=(V_{N,1},\ldots,V_{N,J}):=(I_{k_1}(f_{N,1}),\ldots,I_{k_J}(f_{N,J}))
\]
with $f_{N,j}$ in $L^2(\mathbb{R}^{k_j})$. Let $C$ be a symmetric non-negative definite matrix such that
\[
\mathbb{E}(V_{N,i}V_{N,j})\rightarrow C(i,j).
\]
Then the univariate convergence as $N\rightarrow\infty$
\[
V_{N,j}\ConvD N(0,C(j,j)) \quad j=1,\ldots,J
\]
implies the joint convergence
\[
\mathbf{V}_N\ConvD N(\mathbf{0},C).
\]
\end{Lem}

We now prove \cref{Pure SRD}.

\begin{proof}
Take time points $t_1,\ldots,t_I$, let $\mathbf{V}_N(t)$ be the vector in (\ref{VectorQuest}) in the context of \cref{Pure SRD}, with $G_j$ replaced by a finite linear combination of Hermite polynomials (\cref{ReductCLT}).
Thus
\begin{equation}\label{e:VVect}
\mathbf{V}_N(t_i)=\left(\sum_{m=k_1}^M \frac{g_{m,1}}{A_1(N)}S_{N,t_i}(H_m),\ldots, \sum_{m=k_J}^M \frac{g_{m,J}}{A_J(N)}S_{N,t_i}(H_m)\right).
\end{equation}

We want to show  the joint convergence
\begin{equation}\label{TargetCLT}
\Big(\mathbf{V}_N(t_1),\ldots,\mathbf{V}_N(t_I)\Big)\ConvD\Big(\mathbf{B}(t_1),\ldots,\mathbf{B}(t_I)\Big)
\end{equation}
with $\mathbf{B}(t)$ being the J-dimensional Gaussian process with covariance structure given by (\ref{CovStruct BM}).

By (\ref{Herm<->Int}), and because the term $\frac{g_{m,j}}{A_j(N)}S_{N,t_i}(H_m)$ involves the m-th order Hermite polynomial only, we can represent it as an m-tuple Wiener-It\^{o} integral:
\[
\frac{g_{m,j}}{A_j(N)}S_{N,t_i}(H_m)=:I_m(f_{N,m,i,j})
\]
for some  square-integrable function $f_{N,m,i,j}$. Now
\begin{equation}\label{Write as Integral}
\mathbf{V}_N(t_i)=\left(\sum_{m=k_1}^MI_m(f_{N,m,i,1}),\ldots,\sum_{m=k_J}^MI_m(f_{N,m,i,J})\right)
\end{equation}

To show (\ref{TargetCLT}), one only needs to show that as $N\rightarrow \infty$, $\big(I_m(f_{N,m,i,j})\big)_{m,i,j}$ converges jointly to a multivariate normal  with the correct covariance structure.

Note by the univariate SRD result, namely, \cref{CLT}, each $I_m(f_{N,m,i,j})=\frac{g_{m,j}}{A_j(N)}S_{N,t_i}(H_m)$ converges to a univariate normal. Therefore, by \cref{MultiCLT}, it's sufficient to show the covariance structure of $\big(I_m(f_{N,m,i,j})\big)_{m,i,j}$ is consistent with the covariance structure of $(B_j(t_i))_{i,j}$ as $N\rightarrow\infty$.

Note that $A_j(N)=\sigma_jN^{1/2}$ where $\sigma_j$ is found in (\ref{sigma_j}). If $m_1\neq m_2$,
$$
EI_{m_1}(f_{N,m,i_1,j_1})I_{m_2}(f_{N,m,i_2,j_2})
=
\frac{g_{m_1,j_1},g_{{m_2},j_2}}{\sigma_{j_1}\sigma_{j_2}N}
\mathbb{E}\left(S_{N,t_{i_1}}(H_{m_1})S_{N,t_{i_2}}(H_{m_2})\right)
=0.
$$
If $m_1=m_2=m$,
\begin{align*}
&\mathbb{E}I_{m}(f_{N,m,i_1,j_1})I_{m}(f_{N,m,i_2,j_2})
\\&=
\frac{g_{m,j_1},g_{{m},j_2}}{\sigma_{j_1}\sigma_{j_2}}
\frac{1}{N}\sum_{n_1=1}^{[Nt_{i_1}]}\sum_{n_2=1}^{[Nt_{i_2}]}\mathbb{E}\big(H_m(X_{n_1})H_m(X_{n_2})\big)
\\&=\frac{m!g_{m,j_1},g_{{m},j_2}}{\sigma_{j_1}\sigma_{j_2}}
\frac{1}{N}\sum_{n_1=1}^{[Nt_{i_1}]}\sum_{n_2=1}^{[Nt_{i_2}]}\gamma(n_1-n_2)
\\&\rightarrow \frac{t_{i_1}\wedge t_{i_2}}{\sigma_{j_1}\sigma_{j_2}}g_{m,j_1},g_{{m},j_2}m!\sum_{n=-\infty}^{\infty}\gamma(n)^m~\text{   as } N\rightarrow\infty
\end{align*}
by \cref{ComputeCov}.

Since every component of $\mathbf{V}_N$ in (\ref{e:VVect}) is the sum  of multiple Wiener-It\^{o} integrals, it follows that
\[
\mathbb{E}V_{N,{j_1}}(t_{i_1})V_{N,{j_2}}(t_{i_2})\rightarrow
\frac{t_{i_1}\wedge t_{i_2}}{\sigma_{j_1}\sigma_{j_2}}\sum_{m=k_{j_1}\vee k_{j_2}}^M g_{m,j_1}g_{m,j_2}m!\sum_{n=-\infty}^{\infty}\gamma(n)^m,
\]
which is the covariance in (\ref{CovStruct BM}), where here $M$ is finite due to  \cref{ReductCLT}.
\end{proof}

\subsection{Proof of \cref{Pure LRD}}
\label{Sec:Pure LRD}
The pure LRD case is proved by extending the proof in \cite{dobrushin1979non} to the multivariate case.
Set
\[
S_{N,t}(G)=\sum_{n=1}^{[Nt]}G(X_n).
\]
The normalization factor which makes the variance at $t=1$ tend to 1 is
\begin{equation}\label{e:AjN}
A_j(N)=a_j L(N)^{k_j/2}N^{1+k_j(d-1/2)},
\end{equation}
where the slowly varying function $L(N)$ stems from the auto-covariance function: $\gamma(n)=L(n)n^{2d-1}$ and where $a_j$ is a normalization constant.

The Hermite expansion of each $G_j$ is given in \ref{e:Hexp}
The following reduction lemma shows that it suffices to replace $G_j$'s with corresponding Hermite polynomials.
\begin{Lem}\label{ReductNCLT}
If the convergence in (\ref{Pure LRD Limit}) holds with $g_{k_j,j}H_{k_j}$ replacing $G_j$, then it also holds for $G_j$, $j=1,\ldots,J$.
\end{Lem}
\begin{proof}
By the Cram\'er-Wold device, we want to show for every $(w_1,\ldots,w_J)\in \mathbb{R}^J$, the following convergence:
\[
\sum_{j=1}^J w_j\frac{S_{N,t}(G_j)}{A_j(N)}\ConvFDD\sum_{j=1}^J w_jZ^{(k_j)}_d(t).
\]
Let $G^*_j=g_{k_j+1,j}H_{k_j+1}+g_{k_j+2,j}H_{k_j+2}+\ldots$, then
\[
\sum_{j=1}^J w_j\frac{S_{N,t}(G_j)}{A_j(N)}=\sum_{j=1}^J w_j\frac{S_{N,t}(g_{k_j,j}H_{k_j})}{A_j(N)}+\sum_{j=1}^J w_j\frac{S_{N,t}(G^*_j)}{A_j(N)}.
\]
By the assumption of this lemma and by the Cram\'er-Wold device,
\[
\sum_{j=1}^J w_j\frac{S_{N,t}(g_{k_j,j}H_{k_j})}{A_j(N)}\ConvFDD\sum_{j=1}^J w_jZ^{(k_j)}_d(t).
\]
Hence it suffices to show that for any $t>0$,
\[
\mathbb{E}\left(\sum_{j=1}^J w_j\frac{S_{N,t}(G^*_j)}{A_j(N)}\right)^2\rightarrow 0.
\]
By the elementary inequality: $(\sum_{j=1}^Jx_j)^2\le J\sum_{j=1}^J x_j^2$, it suffices to show that for each $j$,
\[
\mathbb{E}\left(\frac{S_{N,t}(G^*_j)}{A_j(N)}\right)^2\rightarrow 0.
\]
This is because the variance growth of $G^*_j$ (see (\ref{VarGrowthSRD}) and (\ref{VarGrowthLRD})) is at most $L^*_j([Nt])[Nt]^{(k_j+1)(2d-1)+2}$
for some slowly varying function $L_j^*$ , while the normalization
$A_j(N)^2= a_j^2L_j(N)^{k_j}N^{k_j(2d-1)+2}$ tends more rapidly to infinity.
\end{proof}

The following lemma extends Lemma 3 of \cite{dobrushin1979non} to the multivariate case. It states that if Lemma 3 of \cite{dobrushin1979non} holds in the univariate case in each component, then it holds in the multivariate joint case.

\begin{Lem}\label{ConvLemma}
Let  $F_0$ and $F_N$ be symmetric locally finite Borel measures without atoms on $\mathbb{R}$ so that $F_N\rightarrow F$ weakly. Let $W_{F_N}$ and $W_{F_0}$ be complex Hermitian Gaussian measures with control measures $F_N$ and $F_0$ respectively.

Let $K_{N,j}$ be a series of Hermitian($K(-\mathbf{x})=\overline{K(\mathbf{x})}$) measurable functions of $k_j$ variables tending to a continuous  function $K_{0,j}$ uniformly in any compact set in $\mathbb{R}^{k_j}$ as $N\rightarrow\infty$.

Moreover, suppose the following uniform integrability type condition holds for every $j=1,\ldots,J$:
\begin{equation}\label{Integrability}
\lim_{A\rightarrow \infty} \sup_N \int_{\mathbb{R}^{k_j}\setminus [-A,A]^{k_j}} |K_{N,j}(\mathbf{x})|^2 F_N(dx_1),\ldots,F_N(dx_{k_j})=0.
\end{equation}

Then we have the joint convergence:
\begin{equation}\label{ConvLemma Result}
\left(I_{k_1}^{(N)}(K_{N,1}),\ldots,I_{k_J}^{(N)}(K_{N,J})\right)\overset{d}{\rightarrow}
\left(I_{k_1}^{(0)}(K_{0,1}),\ldots,I_{k_J}^{(0)}(K_{0,J})\right).
\end{equation}
where $I_{k}^{(N)}(.)$ denotes a k-tuple Wiener-It\^{o} integral  with respect to complex Gaussian random measure $W_{F_N}$, $N=0,1, 2,\ldots$
\end{Lem}
\begin{proof}
By the Cram\'er-Wold device, we need to show that for every $(w_1,\ldots,w_J)\in \mathbb{R}^J$ as $N\rightarrow\infty$,
\begin{equation}\label{Target}
X_{N}:=\sum_{j=1}^J w_j I_{k_j}^{(N)}(K_{N,j}) \overset{d}{\rightarrow} X_{0,0}:=\sum_{j=1}^J w_j I_{k_j}^{(0)}(K_{0,j}).
\end{equation}

We show first that (\ref{Target}) holds when replacing all kernels with simple Hermitian functions $g_j$ of the form:
\[
g_j(u_1,\ldots,u_{k_j})=\sum_{i_1,\ldots,i_k=1}^na_{i_1,\ldots,i_{k_j}} \mathrm{1}_{A_{i_1,j}\times\ldots\times A_{i_{k_j},j}}(u_1,\ldots,u_{k_j}),
\]
where $A_{i,j}$'s are bounded Borel sets in $\mathbb{R}$ satisfying $F_0(\partial A_{i,j})=0$, $a_{i_1,\ldots,i_{k_j}}=0$ if any two of $i_1,\ldots,i_{k_j}$ are equal, and
$\overline{g(\mathbf{u})}=g(-\mathbf{u})$. We claim that
\begin{equation}\label{Simple g}
\sum_{j=1}^s w_j I_{k_j}^{(N)}(g_j) \overset{d}{\rightarrow} \sum_{j=1}^s w_j I_{k_j}^{(0)}(g_j).
\end{equation}
Indeed, since $F_N\rightarrow F_0$ weakly and $F_0(\partial A_{i,j})=0$, we have as $N\rightarrow\infty$:
\[\mathbb{E}W_{F_N}(A_{i,j})W_{F_N}(A_{k,l})=F_N(A_{i,j}\cap A_{k,l})\rightarrow F_0(A_{i,j}\cap A_{k,l})=\mathbb{E}W_{F_0}(A_{i,j})W_{F_N}(A_{k,l}),
\]
thus
$\big(W_{F_N}(A_{i,j})\big)_{i,j} \overset{d}{\rightarrow} \big(W_{F_0}(A_{i,j})\big)_{i,j}$ jointly . Since $\sum_{j=1}^s w_j I_{k_j}^{(N)}(g_j)$ is a polynomial  of $W_{F_N}(A_{i,j})$ and by Continuous Mapping Theorem, (\ref{Simple g}) holds.

Next, due to the atomlessness of $F_N$, the uniform convergence of $K_{N,j}$ to $K_{0,j}$ on any compact set, (\ref{Integrability}) and the continuity of $K_{0,j}$, for any $\epsilon>0$,  there exist simple Hermitian $g_j$'s $j=1,\ldots,J$ as above, such that for $N=0$ and $N>N(\epsilon)$ (large enough),
\begin{equation}\label{g appro K_N}
\int_{\mathbb{R}^{k_j}}|K_{N,j}(x_1,\ldots,x_{k_j})-g_j(x_1,\ldots,x_{k_j})|^2 F_N(dx_1)\ldots F_N(dx_{k_j})<\epsilon.
\end{equation}
By (\ref{g appro K_N}) for every $j=1,\ldots,J$, we can find a sequence $g_{M,j}$ such that
\begin{equation}\label{e:a1}
\|I^{(0)}_{k_j}(K_{0,j})-I^{(0)}_{k_j}(g_{M,j})\|_{L^2}<1/M,
\end{equation}
\begin{equation}\label{e:a2}
\|I^{N}_{k_j}(K_{N,j})-I^{N}_{k_j}(g_j)\|_{L^2}<1/M~\text{ for $N>N(M)$ (large enough)},
\end{equation}
hence by (\ref{e:a1})
\begin{equation}\label{0,M to 0,0}
X_{0,M}:=\sum_{j=1}^J w_j I^{(0)}_{k_j}(g_{M,j})\overset{d} {\rightarrow} X_{0,0}:=\sum_{j=1}^J w_j I^{(0)}_{k_j}(K_0)\quad \text{as }M\rightarrow\infty.
\end{equation}
and by (\ref{e:a2}),
\begin{align}\label{N to N,M}
&\lim_M \limsup_N \mathbb{E}|X_N-X_{N,M}|^2\notag\\
:=&\lim_M \limsup_N \mathbb{E}\left|\sum_{j=1}^J w_j I_{k_j}^{(N)}(K_{N,j})-\sum_{j=1}^J w_j I_{k_j}^{(N)}(g_{M,j})\right|^2=0.
\end{align}
Finally, replacing $g_j$ by $g_{M,j}$ in ({\ref{Simple g}}), we have
\begin{equation}\label{N,M to 0,M}
X_{N,M}\overset{d}{\rightarrow}X_{0,M}.
\end{equation}
Thus (\ref{Target}), namely, $X_N\ConvD X_{0,0}$, follows now from (\ref{0,M to 0,0}), (\ref{N to N,M}) and (\ref{N,M to 0,M}) and  Theorem 3.2 of \cite{patrick1999convergence}.
\end{proof}

We can now prove \cref{Pure LRD}:
\begin{proof}
Since \cref{ConvLemma} involves only univariate assumptions and concludes with the desired multivariate convergence (\ref{ConvLemma Result}), one needs to treat only the univariate case. This is done in \cite{dobrushin1979non}.
\end{proof}

\subsection{Proof of \cref{SRD&LRD}} \label{Sec:SRD&LRD}

The following result from \cite{nourdin2012asymptotic} will be used:
\begin{Thm}\label{ExtenBiAsymp}
\textbf{(Theorem 4.7 in \cite{nourdin2012asymptotic}.)}
Consider
\begin{align*}
\mathbf{S}_N&=\left(I_{k_{1,S}}(f_{1,S,N}),\ldots,I_{k_{J_S,S}}(f_{J_S,S,N})\right),\\
\mathbf{L}_N&=\left(I_{k_{1,L}}(f_{1,L,N}),\ldots,I_{k_{J_L,L}}(f_{J_L,L,N})\right),
\end{align*}
where $k_{j_S,S}> k_{j_L,L}$ for all $j_S=1,\ldots,J_S$ and $j_L=1,\ldots,J_L$.

Suppose that as $N\rightarrow \infty$, $\mathbf{S}_N$ converges in distribution to a multivariate normal law,  and $\mathbf{L}_N$
converges in distribution to a multivariate law which has moment-determinate components, then there are independent random vectors $\mathbf{Z}$ and $\mathbf{H}$, such that
\[
(\mathbf{S}_N,\mathbf{L}_N)\ConvD (\mathbf{Z},\mathbf{H}).
\]
\end{Thm}
A proof of \cref{ExtenBiAsymp} can be found in Appendix \ref{sec:AsympIndep} (see Theorem \ref{t:B2}).

\begin{proof}[Proof of Theorem \ref{SRD&LRD}]
Using the reduction arguments of \cref{ReductCLT} and \cref{ReductNCLT}, we can replace $G_{j,S}$ in (\ref{SRD Part})  with $\sum_{m={k_{j,S}}}^M g_{m,j,S}H_m$, and we can replace  $G_{j,L}$ in (\ref{LRD Part}) with $g_{k_L,j,L}H_{k_L}$, where $k_{j,S}>k_{j,L}=1 \text{ or }2 $ are the corresponding Hermite ranks and $g_{m,j,S}$, $g_{k_L,j,L}$ are the corresponding coefficients of their Hermite expansions.

Fix finite time points $t_i$, $i=1\ldots,I$, we need to consider the joint convergence of the following vector:
\begin{equation}\label{SRD&LRD Target}
\left(S_{i,j_S,N},L_{i,j_L,N}\right)_{i, j_S, j_L}:=\left( \frac{1}{A_{j_S,S}}\sum_{m={k_{j_S,S}}}^M g_{m,j_S,S}S_{N,t_i}(H_m), \frac{1}{A_{j_L,L}} g_{k_L,j_L,L}S_{N,t_i}(H_{k_L}) \right)_{i, j_S, j_L},
\end{equation}
where $i=1,\ldots,I$, $j_S=1,\ldots, J_S$, $j_L=1,\ldots,J_L$.

As in the proof of \cref{Pure SRD}, using (\ref{Herm<->Int}), we  express Hermite polynomials as multiple Wiener-It\^{o} integrals:
\begin{align*}
S_{i,j_S,N}=\sum_{m=k_{j_S,S}}^MI_m(f_{m,i,j_S,N}),\quad
L_{i,j_L,N}=\sum_{m=k_{j_L,L}}^MI_m(f_{m,i,j_L,N}),
\end{align*}
where $f_{m,i,j_S,N}$, $f_{i,j_L,N}$ are some symmetric square-integrable functions.

Express the vector in (\ref{SRD&LRD Target})
 as $(\mathbf{S}_N,\mathbf{L}_N)$, where $\mathbf{S}_N:=(S_{i,j_S,N})_{i,j_S}$,
$\mathbf{L}_N:=(L_{i,j_L,N})_{i,j_L}$.

By \cref{Pure SRD}, $\mathbf{S}_N$ converges in distribution to some multivariate normal distribution, and by \cref{Pure LRD}, $\mathbf{L}_N$ converges to a multivariate distribution with moment-determinate marginals, because by assumption the limits only involve Hermite rank $k=1$ (normal distribution) and $k=2$ (Rosenblatt distribution). The normal distribution is moment-determinate. The Rosenblatt distribution is also moment-determinate because it has analytic characteristic function (\cite{taqqu1975weak} p.301).

We can now use \cref{ExtenBiAsymp} to conclude the proof.
\end{proof}

\appendix
\section{Invariance of joint distribution among different representations of Hermite process}\label{sec:JointInvar}
The Hermite process admits four different representations (\cite{pipiras2010regularization}):

Let $B(.)$ be the real Gaussian random measure and $W(.)$ be the complex Gaussian  random measure, as defined  in Section 6 of \cite{taqqu1979convergence}. $H_0\in (1-1/(2k),1)$.
\begin{enumerate}
\item Time domain representation:
\begin{equation}\label{eq:TimeDomRep}
Z^{(k)}_{H_0}(t)=a_{k,H_0} =\int'_{\mathbb{R}^k}\left(\int_0^t \prod_{j=1}^k(s-x_j)_+^{H_0-3/2}ds\right)B(dx_1)\ldots B(dx_k)
\end{equation}
\item Spectral domain representation:
\begin{equation}\label{eq:SpecDomRep}
Z^{(k)}_{H_0}(t)=b_{k,H_0}\int''_{\mathbb{R}^k}\frac{e^{i(x_1+\ldots+x_k)t}-1}{i(x_1+\ldots+x_k)}\prod_{j=1}^k |x_j|^{1/2-H_0}W(dx_1)\ldots W(dx_k)
\end{equation}
\item Positive half-axis representation:
\begin{equation}\label{eq:PosAxiRep}
Z^{(k)}_{H_0}(t)=c_{k,H_0}\int'_{[0,\infty)^k}\left(\int_0^t\prod_{j=1}^k x_j^{1/2-H_0}(1-sx_j)_+^{H_0-3/2}ds\right)B(dx_1)\ldots B(dx_k)
\end{equation}
\item Finite interval representation:
\begin{equation}\label{eq:FinIntRep}
Z^{(k)}_{H_0}(t)=d_{k,H_0}\int'_{[0,t]^k}\left(\prod_{j=1}^k x_j^{1/2-H_0}\int_0^t x^{k(H_0-1/2)}\prod_{j=1}^k(s-x_j)_+^{H_0-3/2}ds\right)B(dx_1)\ldots B(dx_k)
\end{equation}
\end{enumerate}
where  $a_{k,H_0},b_{k,H_0},c_{k,H_0},d_{k,H_0}$ are constant coefficients to guarantee that $\mathrm{Var}(Z^{(k)}_{H_0}(t))=1$,  given in (1.17) and (1.18) of \cite{pipiras2010regularization}.

Keep $H_0$ fixed throughout. We will prove the following:
\begin{Thm}\label{t:order}
The joint distribution of a vector made up of Hermite processes of possibly different orders $k$, but sharing the same random measure $B(.)$ or $W(.)$ in their Wiener-It\^o integral representations,  remains the same when  switching from one of the above representations to another.
\end{Thm}

The  following  notations are used to denote Wiener-It\^o integrals with respect to $B(.)$ and $W(.)$ respectively: \[I(f):=\int'_{\mathbb{R}^k}f(x_1,\ldots,x_k)dB(x_1)\ldots dB(x_k),\]
\[\tilde{I}(g):=\int''_{\mathbb{R}^k} g(\omega_1,\ldots,\omega_k)dW(\omega_1)\ldots dW(\omega_k).\]
where $'$ indicates that we don't integrate on $x_i=x_j, i\neq j$, $''$ indicates that we don't integrate on $\omega_i=\pm \omega_j$, $ i\neq j$, $f$ is a symmetric function and $g$ is an Hermitian function ($g(\mathbf{\omega})=\overline{g(-\mathbf{\omega})}$).

The next lemma establishes the equality in joint distribution between time domain representation (\ref{eq:TimeDomRep}) and spectral domain representation  (\ref{eq:SpecDomRep}), which  is a multivariate extension of Lemma 6.1 in \cite{taqqu1979convergence}.
\begin{Lem}
Suppose that $A_j(x_1,\ldots,x_{k_j})$ is a symmetric function in $L^2(\mathbb{R}^{k_j})$, $j=1,\ldots,J$. Let $\tilde{A}(x_1,\ldots,x_{k_j})$ be its $L^2$-Fourier transform:
\[
\tilde{A}_j(\omega_1,\ldots,\omega_{k_j})
=\frac{1}{(2\pi)^{k_j/2}}\int_{\mathbb{R}^m} \exp(i\sum_{n=1}^{k_j}x_n\omega_n)A_j(x_1,\ldots,x_{k_j})dx_1\ldots dx_{k_j}.
\]
Then
\[
\left(I_{k_1}(A_1),\ldots,I_{k_J}(A_J)\right)
\overset{d}{=}\left(\tilde{I}_{k_1}(\tilde{A}_1),\ldots,\tilde{I}_{k_J}(\tilde{A}_J)\right).
\]
\end{Lem}
\begin{proof}
The proof is a slight extension of the proof of Lemma 6.1 of \cite{taqqu1979convergence}. The idea is to use a complete orthonormal  set $\{\psi_i,i\ge 0\}$ in $L^2(\mathbb{R})$ to represent each $A_j$ as an infinite polynomial form of order $k_j$ with respect to $\psi_i$'s, as is done in (6.3) of \cite{taqqu1979convergence}. Each $I_{k_j}(A_j)$ can be then written in the form of (6.4) of \cite{taqqu1979convergence}, which  is essentially a  function of   $X_i:=\int\psi_i(x)dB(x), i\ge 0$, denoted
\begin{align*}
I_{k_j}(A_j)=K_j(\mathbf{X}),
\end{align*}
where $\mathbf{X}=(X_0,X_1,\ldots)$. Thus
\begin{align}\label{eq:K_j(X)}
\left(I_{k_1}(A_1),\ldots,I_{k_J}(A_J)\right)=\mathbf{K}(\mathbf{X}),
\end{align}
where the vector function $\mathbf{K}=(K_1,\ldots,K_J)$.

Now, $\tilde{A}_j$ can also be written as an infinite polynomial form of order $k_j$ with respect to $\tilde{\psi}_i, i\ge 0$, where $\tilde{\psi}_i(\omega)= (2\pi)^{-1/2}\int e^{ix\omega}\psi_i(x)dx$ is the $L^2$-Fourier transform of $\psi_i$, as is done in (6.5) of \cite{taqqu1979convergence}. Set $Y_j:=\int\tilde{\psi}_i(\omega)dW(\omega), i\ge 0$. Then, as in (6.6) of \cite{taqqu1979convergence}, we have
\begin{align*}
\tilde{I}_{k_j}(\tilde{A}_j)=K_j(\mathbf{Y}),
\end{align*} where $K_j$'s are the same as above, $\mathbf{Y}=(Y_0,Y_1,\ldots)$, and thus
\begin{align}\label{eq:K_j(Y)}
\left(\tilde{I}_{k_1}(\tilde{A}_1),\ldots,\tilde{I}_{k_J}(\tilde{A}_J)\right)=\mathbf{K}(\mathbf{Y}).
\end{align}

By (\ref{eq:K_j(X)}) and (\ref{eq:K_j(Y)}), it suffices to show that $\mathbf{X}\overset{d}{=}\mathbf{Y}$. This is true because by Parseval's identity,  $\mathbf{X}$ and $\mathbf{Y}$ both consist of i.i.d. normal random variables with mean 0 and identical variance, . For details, see \cite{taqqu1979convergence}.
\end{proof}

We now  complete the proof of \cref{t:order}. We still need to justify the equality in joint distribution between time domain representation (\ref{eq:TimeDomRep}) and positive half-axis representation (\ref{eq:PosAxiRep}) or finite interval representation (\ref{eq:FinIntRep}).

First let's summarize the arguments of \cite{pipiras2010regularization} for going from (\ref{eq:TimeDomRep}) to (\ref{eq:PosAxiRep}) or (\ref{eq:FinIntRep}).
The heuristic idea is that by changing the integration order in (\ref{eq:TimeDomRep}), one would have 
\begin{align}
Z^{(k)}_{H_0}&=\int_{0}^t\left( \int'_{\mathbb{R}^k} \prod_{j=1}^k (s-x_j)^{H_0-3/2} B(dx_1)\ldots B(dx_k)\right) ds\notag\\
&= \int_0^t H_k\left(\int_\mathbb{R} (s-x)_+^{H_0-3/2}B(dx)\right)ds\label{eq:HeuChangeOrder},
\end{align}
where $H_k$ is $k$-th Hermite polynomial. But in fact $g(x):=(s-x)_+^{H_0-3/2}\notin L^2(\mathbb{R})$, and consequently $G(s):=\int_\mathbb{R} (s-x)_+^{H_0-3/2}B(dx)$  is not well-defined.

The way to get around this is to do a regularization, that is, to truncate $g(x)$ as $g_\epsilon(x):=g(x)1_{s-x>\epsilon}(x)$ for $\epsilon>0$. Now the Gaussian process $G_{\epsilon}(t):=\int_\mathbb{R}g_\epsilon(x) B(dx)$ is well-defined. Next, after some change of variables, one gets the new desired representation of $G_{\epsilon}(t)$, say $G^*_{\epsilon}(t)$, where $G^*_{\epsilon}(t)\overset{d}{=}G_{\epsilon}(t)$. Setting $Z^{(k)}_{\epsilon,H_0}(t)=\int_0^t H_k(G_\epsilon(t))dt$ and
$Z^{(k)*}_{\epsilon,H_0}(t)=\int_0^t H_k(G^*_\epsilon(t))dt$, yields
\begin{equation}\label{e:Zk}
 Z^{(k)}_{\epsilon,H_0}(t)\overset{d}{=}Z^{(k)*}_{\epsilon,H_0}(t).
 \end{equation}
  Finally by letting $\epsilon\rightarrow 0$, one can show that $Z^{(k)}_{\epsilon,H_0}(t)$ converges in $L^2(\Omega)$ to the Hermite process $Z^{(k)}_{H_0}(t)$, while $Z^{(k)*}_{\epsilon,H_0}(t)$ converges in $L^2(\Omega)$ to some $Z^{(k)*}_{H_0}(t)$, which is  the desired alternative representation of $Z^{(k)}_{H_0}(t)$.

The above argument relies on the stochastic Fubini theorem (Theorem 2.1 of \cite{pipiras2010regularization}) which legitimates the change of integration order, that is, for $f(s,\mathbf{x})$ defined on $\mathbb{R}\times \mathbb{R}^k$,
if
$\int_\mathbb{R} \|f(s,.)\|_{L^2(\mathbb{R}^k)}ds <\infty$
(which is the case after regularization), then
\begin{align*}
\int'_{\mathbb{R}^k}\int_\mathbb{R} f(s,x_1,\ldots,x_k) ds B(dx_1)\ldots B(dx_k)
=\int_\mathbb{R}\int'_{\mathbb{R}^k}f(s,x_1,\ldots,x_k) B(dx_1)\ldots B(dx_k)ds\quad a.s.
\end{align*}

Now, consider the multivariate case.  Note that we still have equality of the the joint distributions
as in (\ref{e:Zk}) and the equality is preserved in the $L^2(\Omega)$ limit as $\epsilon \rightarrow 0$. Moreover,
  the stochastic Fubini theorem (Theorem 2.1 of \cite{pipiras2010regularization}) extends naturally to the multivariate setting since the change of integration holds as an almost sure equality.
Therefore one gets equality in joint distribution when  switching from (\ref{eq:TimeDomRep}) to (\ref{eq:PosAxiRep}) or (\ref{eq:FinIntRep}). \hfill $\square$

\section{Asymptotic independence of Wiener-It\^o integral vectors}\label{sec:AsympIndep}
We prove here Theorem 4.1 by extending a combinatorial proof of Nourdin and Rosinski \cite{nourdin2011asymptotic}
\footnote{The  proof in Theorem \ref{AsympIndep} below is an extension to Wiener-It\^o integral vectors of the original combinatorial proof of Theorem 3.1 of \cite{nourdin2011asymptotic} given for Wiener-It\^o integral scalars. The result also follows  from Theorem 3.4
in  \cite{nourdin2012asymptotic} which includes  Wiener-It\^o integral vectors, but with a proof based on Malliavin Calculus.
}.

First, some background.
In the papers \cite{ustunel1989independence} and \cite{kallenberg1991independence}, a criterion for independence  between two random variables belonging to Wiener Chaos, say, $I_p(f)$ and $I_p(g)$, is given as
\begin{equation}\label{Indep}
f{\otimes_1} g=0 \qquad a.s.
\end{equation}
where $\otimes_1$ means contraction of order 1 and is defined below.

 The result of \cite{nourdin2012asymptotic} involves the following problem: if one has sequences $\{f_n\}$, $\{g_n\}$, when will asymptotic independence hold between $I_p(f_n)$ and $I_q(g_n)$ as $n\rightarrow\infty$?
 Motivated by (\ref{Indep}), one may guess that the criterion is $f_n{\otimes_1} g_n\rightarrow 0$ as $n\rightarrow \infty$. This is, however, shown to be false by a counterexample in \cite{nourdin2012asymptotic}: set $p=q=2$, $f_n=g_n$ and assume that $I_2(f_n)\ConvD Z\sim N(0,1)$. One can then show that $f_n\otimes_1 f_n\rightarrow 0$, while obviously  $\left(I_2(f_n),I_2(f_n)\right)\ConvD (Z,Z)$.
Let $\|.\|$ denote the $L^2$ norm in the appropriate dimension and let $<.,.>$ denote the corresponding inner product.

We now define contractions.
The contraction $\otimes_r$ between two symmetric square integrable functions $f$ and $g$ is defined as
\begin{align*}
&(f\otimes_r g)(x_1,\ldots,x_{p-r},y_1,\ldots,y_{q-r}):=\\&\int_{\mathbb{R}^r}f(x_1,\ldots,x_{p-r},s_1,\ldots,s_r)g(y_1,\ldots,y_{q-s},s_1,\ldots,s_r)ds_1\ldots ds_r
\end{align*}
If $r=0$, the contraction is just the tensor product:
\begin{equation}\label{e:TensorProd}
f\otimes_0 g=f\otimes g:=f(x_1,\ldots,x_{p})g(y_1,\ldots,y_{q}).
\end{equation}

The symmetrized contraction $\tilde{\otimes}_r$ involves one more step, namely, the symmetrization of the function obtained from the contraction. This is done by summing over all permutations of the variables and dividing by the number of permutations. Note that as the contraction is only defined for symmetric functions,  replacing ${\otimes}_r$ with $\tilde{\otimes}_r$ enables one to consider a sequence of symmetrized contractions of the form $\Big(\ldots\big((f_1\tilde{\otimes}_{r_1}f_2)\tilde{\otimes}_{r_2}f_3\big)\ldots\Big)\tilde{\otimes}_{r_{n-1}}f_n$
.

We will use the following product formula (Proposition 6.4.1 of \cite{peccati2011wiener}) for multiple Wiener-It\^{o} integrals
\begin{equation}\label{ProductFormulaPre}
I_p(f)I_q(g)=\sum_{r=0}^{p\wedge q} r! \binom{p}{r}\binom{q}{r}I_{p+q-2r}(f{\otimes}_r g) \quad p,q\ge0.
\end{equation}
Because the symmetrization of the integrand doesn't change the multiple Wiener-It\^{o} integral, ${\otimes}_r$ could be replaced with $\tilde{\otimes}_r$ in the product formula.

For a vector $\mathbf{q}=(q_1,\ldots,q_k)$, we denote $|\mathbf{q}|:=q_1+\ldots+q_k$. By a suitable iteration of (\ref{ProductFormulaPre}), we have the following multiple product formula:
\begin{equation}\label{ProductFormula}
\prod_{i=1}^k I_{q_i}(f_i)=\sum_{\mathbf{r}\in C(\mathbf{q},k)} a(\mathbf{q},k,\mathbf{r})I_{|\mathbf{q}|-2|\mathbf{r}|}\left( \ldots(f_1\tilde{\otimes}_{r_1}f_2)\ldots \tilde{\otimes}_{r_{k-1}}f_k \right),
\end{equation}
where $\mathbf{q}\in \mathbb{N}^n$,  the index set $C(\mathbf{q},k)=\{\mathbf{r}\in \prod_{i=1}^{k-1}\{0,1,\ldots,q_{i+1}\}: r_1\le q_1, r_i\le (q_1+\ldots+q_i)-2(r_1+\ldots+r_{i-1}),i=2,\ldots k-1\}$, and $a(\mathbf{q},k,\mathbf{r})$ is some integer factor.

\begin{Thm}\label{AsympIndep}
\textbf{(Asymptotic Independence of Multiple Wiener-It\^{o} Integral Vectors.)}
Suppose we have the joint convergence
\[
(\mathbf{U}_{1,N},\ldots,\mathbf{U}_{J,N}) \ConvD(\mathbf{U}_1,\ldots,\mathbf{U}_J),
\]
where
\[
\mathbf{U}_{j,N}=\left(I_{q_{1,j}}(f_{1,j,N}),\ldots,I_{q_{I_j,j}}(f_{I_j,j,N}) \right).
\]
Assume
\begin{equation}\label{Contraction2Zero}
\lim_{N\rightarrow\infty} \|f_{i_1,j_1,N}\otimes_{r}f_{i_1,j_2,N}\|=0
\end{equation}
for all $i_1,i_2, j_1\neq j_2$, and $r=1,\ldots,q_{i_1,j_1}\wedge q_{i_2,j_2}$.

Then using the notation $\mathbf{u}^{\mathbf{k}}=u_1^{k_1}\ldots u_m^{k_m}$, we have
\begin{equation}\label{MomentAsympIndep}
\mathbb{E}[\mathbf{U}_1^{\mathbf{k}_1}\ldots\mathbf{U}_J^{\mathbf{k}_J}]
=\mathbb{E}[\mathbf{U}_1^{\mathbf{k}_1}]\ldots \mathbb{E}[\mathbf{U}_J^{\mathbf{k}_J}]
\end{equation}
for all $\mathbf{k}_j\in \mathbb{N}^{I_j}$

Moreover, if every component of every $\mathbf{U}_{j}$ is  moment-determinate, then
$\mathbf{U}_1,\ldots,\mathbf{U}_J$ are independent.
\end{Thm}

\begin{proof}
This is an extension of a proof in \cite{nourdin2011asymptotic}.

The index $i=1,\ldots,I_j$ refers to the components within the vector $\mathbf{U}_{j,N}$, $j=1,\ldots,J$. For notational simplicity, we let $I_j=I$, that is, each $\mathbf{U}_{j,N}$ has the same number of components.

Let $|\mathbf{k}|$ denote the sum of its components $k_1+\ldots+k_m$.
First  to show  (\ref{MomentAsympIndep}), it suffices to show
\[
\lim_{N\rightarrow\infty} \mathbb{E}\prod_{j=1}^J (\mathbf{U}_{j,N}^{\mathbf{k}_j}-\mathbb{E}[\mathbf{U}_{j,N}^{\mathbf{k}_j}])=0
\]
for any $|\mathbf{k}_1|>0,\ldots,|\mathbf{k_J}|>0$. Note that $\mathbf{U}_{j,N}^{\mathbf{k}_j}=U_{1,j,N}^{k_{1,j}}\ldots U_{I,j,k}^{k_{I,j}}$ is a scalar.

By (\ref{ProductFormula}), one gets
\[
I_q(f)^k=\sum_{\mathbf{r}\in C_{q,k}}a(q,k,\mathbf{r})I_{kq-2|\mathbf{r}|}
\left(\ldots(f\tilde{\otimes}_{r_1} f)\ldots \tilde{\otimes}_{r_{k-1}}f\right)
\]
where $a(q,k,r)$'s are integer factors which don't play an important role, and $C_{q,k}$ is some index set.
If $\mathbf{U}_{j,N}^{\mathbf{k}_j}=\prod_{i=1}^I I_{q_{i,j}}(f_{i,j,N})^{k_{i,j}}$, then
\begin{align}\label{Product1}
\mathbf{U}_{j,N}^{\mathbf{k}_j}
&=\prod_{i=1}^I \sum_{\mathbf{r}\in C_{q_{i,j},k_{i,j}}}a(q_{i,j},k_{i,j},\mathbf{r})I_{k_{i,j}q_{i,j}-2|\mathbf{r}|}
\left(\ldots(f_{i,j,N}\tilde{\otimes}_{r_1} f_{i,j,N})\ldots \tilde{\otimes}_{r_{k_{i,j}-1}}f_{i,j,N}\right) \nonumber \\
&=\sum_{\mathbf{r}^1\in C_{q_{1,j},k_{1,j}}}\ldots\sum_{\mathbf{r}^I\in C_{q_{I,j},k_{I,j}}}  \prod_{i=1}^I
a(q_{i,j},k_{i,j},\mathbf{r}^i)I_{k_{i,j}q_{i,j}-2|\mathbf{r}^i|}(h_{i,j,N})
\end{align}
where
\[h_{i,j,N}=\left(\ldots(f_{i,j,N}\tilde{\otimes}_{r^i_1} f_{i,j,N})\ldots \tilde{\otimes}_{r^i_{k_{i,j}-1}}f_{i,j,N}\right).\]

If one applies the product formula (\ref{ProductFormula}) to the product in (\ref{Product1}), one gets that $\mathbf{U}_{j,N}^{\mathbf{k}_j}$ involves terms of the form $I_{|\mathbf{p}_j|-2|\mathbf{s}_j|}(H_{j,N})$ ($\mathbf{p}_j$ and $\mathbf{s}_j$ run through some suitable index sets),
where
\[
H_{j,N}=\left(\ldots(h_{1,j,N}\tilde{\otimes}_{s_1} h_{2,j,N})\ldots \tilde{\otimes}_{s_{I-1}}h_{I,j,N}\right).
\]
Since the expectation of a Wiener-It\^o integral of positive order is 0 while a Wiener-It\^o integral of zero order is a constant, $\mathbf{U}_{j,N}^{\mathbf{k}_j}-\mathbb{E}[\mathbf{U}_{j,N}^{\mathbf{k}_j}]$ involves $I_{|\mathbf{p}_j|-2|\mathbf{s}_j|}(H_{j,N})$  with $|\mathbf{p}_j|-2|\mathbf{s}_j|>0$ only. Therefore, every $H_{j,N}$ involved in the expression of $\mathbf{U}_{j,N}^{\mathbf{k}_j}-\mathbb{E}[\mathbf{U}_{j,N}^{\mathbf{k}_j}]$ has $n_j=|\mathbf{p}_j|-2|\mathbf{s}_j|>0$ variables.

Note that there are no products left at this point in the expression of $\mathbf{U}_{j,N}^{\mathbf{k}_j}-\mathbb{E}[\mathbf{U}_{j,N}^{\mathbf{k}_j}]$, only sums.
But to compute $\mathbb{E}\prod_{j=1}^J (\mathbf{U}_{j,N}^{\mathbf{k}_j}-\mathbb{E}[\mathbf{U}_{j,N}^{\mathbf{k}_j}])
$, one needs to apply the product formula (\ref{ProductFormula}) again and then compute the expectation. Since  Wiener-It\^{o} integrals of positive order have mean 0, taking the expectation involves focusing on the terms of zero order which are constants. Since $f\otimes_p g=<f,g>=EI_p(f)I_p(g)$ for functions $f$ and $g$ both having $p$ variables, $\mathbb{E}\prod_{j=1}^J (\mathbf{U}_{j,N}^{\mathbf{k}_j}-\mathbb{E}[\mathbf{U}_{j,N}^{\mathbf{k}_j}])$ involves only terms of the form:
\begin{align}
G_N&=\left(\ldots(H_{1,N}\tilde{\otimes}_{t_1} H_{2,N})\ldots\tilde{\otimes}_{t_{J-2}}H_{J-1,N}\right)\tilde{\otimes}_{t_{J-1}} H_{J,N}\\
&=\int_{\mathbb{R}^{n_J}} \left(H_{1,N}\tilde{\otimes}_{t_1} H_{2,N})\ldots\tilde{\otimes}_{t_{J-2}}H_{J-1,N}\right)H_{J,N}~ d\mathbf{x}\label{e:G_N}
\end{align}
where the contraction size vector $\mathbf{t}=(t_1,\ldots,t_{J-1})$ runs through some index set. Since these contractions must yield a constant, we have
\begin{equation}\label{e:|t|}
|\mathbf{t}|=\frac{1}{2}(n_1+\ldots+n_J)>0,
\end{equation}
where $n_j$ is the number of variables of $H_{j,N}$. There is therefore at least one component (call it $t$) of $\mathbf{t}$ which is strictly positive and thus there is a pair $j_1, j_2$ with $j_1 \neq j_2$, such that $H_{J_1}$ and $H_{j_2}$ that have at least one common argument.

One now needs to show that $G_N$ in (\ref{e:G_N}) tends to 0. This is done by applying the generalized Cauchy-Schwartz inequalities in Lemma 2.3 of \cite{nourdin2012asymptotic} successively, through the following steps:
\begin{align}
&\text{for any $j_1\neq j_2$, $i_1,i_2$ and $r>0$, }  \lim_{N\rightarrow\infty}\|f_{i_1,j_1,N}\otimes_r f_{i_2,j_2,N}\| = 0~ \notag\\&\implies
\text{for any $j_1\neq j_2$, $i_1,i_2$ and $s>0$, }
\lim_{N\rightarrow\infty}\|h_{i_1,j_1,N}\otimes_s h_{i_2,j_2,N}\|=0~  \notag\\&\implies
\text{for any $j_1\neq j_2$ and $t>0$, }
\lim_{N\rightarrow\infty}\|H_{j_1,N}\otimes_t H_{j_2,N}\|=0~ \label{e:H}\\&\implies
\lim_{N\rightarrow\infty}G_N=0,\label{e:G}
\end{align}
proving (\ref{MomentAsympIndep}). Here we illustrate some details for going from (\ref{e:H}) to (\ref{e:G}), and omit the first two steps which use a similar argument.

Let $C=\{1,2,\ldots,(n_1+\ldots n_J)/2\}$. Suppose $c$ is a subset of $C$, then we use the notation $\mathbf{z}_c$ to denote $\{z_{j_1},\ldots,z_{j_{|c|}}\}$ where $\{j_1,\ldots,j_{|c|}\}=c$ and $|c|$ is the cardinality of $c$. When $c=\emptyset$, $\mathbf{z}_c=\emptyset$.

Observe that (\ref{e:G_N}) is a sum (due to symmetrization) of terms of the form:
\begin{equation}\label{e:G_N Detail}
\int_{\mathbb{R}^{|C|}} H_{1,N}(\mathbf{z}_{c_1})\ldots H_{J,N}(\mathbf{z}_{c_J})d\mathbf{z}_C,
\end{equation}
where every $c_j$, $j=1,\ldots,J$, is a subset of $C$. Note that  since $|t|=t_1+\ldots+t_J>0$ in (\ref{e:|t|}), there must exist $j_1\neq j_2 \in \{1,\ldots,J\}$, such that $c_0:=c_{j_1}\cap c_{j_2}\neq \emptyset$.
By the generalized Cauchy Schwartz inequality (Lemma 2.3 in \cite{nourdin2012asymptotic}), one gets a bound for (\ref{e:G_N Detail}) as:
\begin{align*}
\left|\int_{\mathbb{R}^{|C|}} H_{1,N}(\mathbf{z}_{c_1})\ldots H_{J,N}(\mathbf{z}_{c_J})d\mathbf{z}_C\right| \le
\|H_{j_1,N} \otimes_{|c_0|} H_{j_2,N}\| \prod_{j\neq j_1,j_2} \|H_{j,N}\|,
\end{align*}
where $\|H_{j_1,N} \otimes_{|c_0|} H_{j_2,N}\|\rightarrow 0$ as $N\rightarrow \infty$ by (\ref{e:H}). In addition, $\|f_{i,j,N}\|$, $N\ge 1$  are uniformly  bounded due to the tightness of the distribution of $I_{k_{i,j}}(f_{i,j,N}), N\ge 1$ (Lemma 2.1 of \cite{nourdin2012asymptotic}). This, by the generalized Cauchy-Schwartz inequality (Lemma 2.3 of in \cite{nourdin2012asymptotic}), implies that $\|h_{i,j,N}\|, N\ge 1$ are uniformly bounded, which further implies the uniform boundedness of  $\|H_{j,N}\|, N\ge 1$.
Hence (\ref{e:G_N Detail}) goes to $0$ as $N\rightarrow\infty$ and thus (\ref{e:G}) holds.

Finally, if every component of every $\mathbf{U}_{j}$ is  moment-determinate, then by Theorem 3 of
\cite{petersen1982relation}, the distribution of $\mathbf{U}:=(\mathbf{U}_1,\ldots,\mathbf{U}_J)$ is determined by its joint moments. But by (\ref{MomentAsympIndep}), the joint moments of $\mathbf{U}$ are the same as if the $\mathbf{U}_j$'s were independent. Then the joint moment-determinancy implies independence.
\end{proof}

\begin{Cor}\label{CorForAsympIndep}
With the notation of \cref{AsympIndep}, suppose that condition (\ref{Contraction2Zero}) is satisfied and that as $N\rightarrow \infty$, each $\mathbf{U}_{j,N}$ converges in distribution to some multivariate law which has moment-determinate components. Then there are independent random vectors $\mathbf{U}_1,\ldots,\mathbf{U}_J$ such that
\begin{equation}\label{U Conv Jointly}
(\mathbf{U}_{1,N},\ldots,\mathbf{U}_{J,N}) \ConvD(\mathbf{U}_1,\ldots,\mathbf{U}_J).
\end{equation}
\end{Cor}
\begin{proof}
Since each $\mathbf{U}_{j,N}$ converges in distribution, the vector of vectors $(\mathbf{U}_{1,N},\ldots,\mathbf{U}_{J,N})$ is  tight in distribution, so any of its subsequence has a further subsequence converging in distribution to a vector
$(\mathbf{U}_1,\ldots\mathbf{U}_J)$. But by \cref{AsympIndep},  the $\mathbf{U}_j$'s are independent, and the convergence in distribution of each $\mathbf{U}_{j,N}$ implies that $\mathbf{U}_{j,N}\ConvD \mathbf{U}_{j}$, and hence (\ref{U Conv Jointly}) holds.
\end{proof}

Now we are in the position to state the result used in Theorem
\ref{SRD&LRD} in the proof of the SRD and LRD mixed case.
\begin{Thm} \label{t:B2}
Consider
\begin{align*}
\mathbf{S}_N&=\left(I_{k_{1,S}}(f_{1,S,N}),\ldots,I_{k_{J_S,S}}(f_{J_S,S,N})\right),\\
\mathbf{L}_N&=\left(I_{k_{1,L}}(f_{1,L,N}),\ldots,I_{k_{J_L,L}}(f_{J_L,L,N})\right),
\end{align*}
where $k_{j_S,S}> k_{j_L,L}$ for all $j_S=1,\ldots,J_S$ and $j_L=1,\ldots,J_L$.

Suppose that as $N\rightarrow \infty$, $\mathbf{S}_N$ converges in distribution to a multivariate normal law,  and $\mathbf{L}_N$
converges in distribution to a multivariate law which has moment-determinate components, then there are independent random vectors $\mathbf{Z}$ and $\mathbf{H}$, such that
\[
(\mathbf{S}_N,\mathbf{L}_N)\ConvD (\mathbf{Z},\mathbf{H}).
\]
\end{Thm}
\begin{proof}
By \cref{CorForAsympIndep}, we only need to check the contraction condition (\ref{Contraction2Zero}). This is done  as in the proof of Theorem 4.7 of \cite{nourdin2012asymptotic}. For the convenience of the reader, we present the argument here.

Using the identity $\|f \otimes_r g\|^2=<f\otimes_{p-r}f,g\otimes_{q-r} g>$ where $r=1,\ldots, p\wedge q$, $f$ and $g$ have respectively $p$ and $q$ variables, we get for $r=1,\ldots,k_{i,L}$,
\begin{align*}
\|f_{i,S,N}\otimes_r f_{j,L,N}\|^2
&=<f_{i,S,N}\otimes_{k_{i,S}-r}f_{j,S,N},f_{j,L,N}\otimes_{k_{j,L}-r}f_{j,L,N}>\\
&\le \|f_{i,S,N}\otimes_{k_{i,S}-r}f_{j,S,N}\| \|f_{j,L,N}\otimes_{k_{j,L}-r}f_{j,L,N}\|\rightarrow 0
\end{align*}
because $\|f_{i,S,N}\otimes_{k_{i,S}-r}f_{j,S,N}\|\rightarrow 0$ by the Nualart-Peccati Central Limit Theorem \cite{nualart2005central}, and for the second term, one has by Cauchy-Schwartz inequality, $\|f_{j,L,N}\otimes_{k_{j,L}-r}f_{j,L,N}\|\le \|f_{j,L,N}\|^2$ (generalized Cauchy-Schwartz inequality  in  \cite{nourdin2012asymptotic} Lemma 2.3), which is bounded due to the tightness of the distribution of $I_{k_{j,L}}(f_{j,L,N})$ (Lemma 2.1 of \cite{nourdin2012asymptotic}).
Therefore (\ref{Contraction2Zero}) holds and the conclusion follows from \cref{CorForAsympIndep}.
\end{proof}

\noindent\textbf{Acknowledgements.} We would like to thank the referee for the careful reading of the paper and the comments. This work was partially supported by the NSF grant DMS-1007616 at Boston University.

\bibliographystyle{plain}

\bigskip

\noindent Shuyang Bai and Murad S. Taqqu\\
Department of Mathematics and Statistics\\
111 Cumminton Street\\
Boston, MA, 02215, US\\
\textit{bsy9142@bu.edu}\\
\textit{murad@bu.edu}
\end{document}